\documentclass[11pt]{amsart}
\usepackage{graphicx}
\usepackage{graphics}
\usepackage{amsmath}
\usepackage{amscd}
\usepackage{latexsym}
\usepackage[all]{xy}
\begin{document}
\textwidth 5.5in
\textheight 8.3in
\evensidemargin .75in
\oddsidemargin.75in

\newtheorem{lem}{Lemma}[section]
\newtheorem{fact}[lem]{Fact}
\newtheorem{conj}[lem]{Conjecture}
\newtheorem{defn}[lem]{Definition}
\newtheorem{thm}[lem]{Fact}
\newtheorem{cor}[lem]{Corollary}
\newtheorem{prob}[lem]{Problem}
\newtheorem{claim}[lem]{Claim}
\newtheorem{main}{Theorem}
\newtheorem{exm}[lem]{Example}
\newtheorem{rmk}[lem]{Remark}
\newtheorem{que}[lem]{Question}
\newtheorem{prop}[lem]{Proposition}
\newtheorem{clm}[lem]{Claim}
\newcommand{\p}[3]{\Phi_{p,#1}^{#2}(#3)}
\def\Z{\mathbb Z}
\def\C{\mathcal{C}}
\def\D{\mathcal{D}}
\def\P{\mathcal{P}}
\def\R{\mathbb R}
\def\g{\text{\normalfont{gal}}}
\def\odots{\reflectbox{\text{$\ddots$}}}
\newcommand{\tg}{\overline{g}}
\def\proof{{\bf Proof. }}
\def\ee{\epsilon_1'}
\def\ef{\epsilon_2'}
\title{Notes on Gompf's infinite order corks}
\author{Motoo Tange}
\thanks{The author is supported by JSPS KAKENHI Grant Number 26800031.}
\address{Institute of Mathematics, University of Tsukuba,
 1-1-1 Tennodai, Tsukuba, Ibaraki 305-8571, Japan}
\email{tange@math.tsukuba.ac.jp}
\date{\today}
\maketitle
%%%%%%%%%%%%%%%%%%%%%%%%%%%%%%%%%%%%%
\begin{abstract}
For any positive integer $n$ we give a ${\mathbb Z}^n$-cork with a ${\mathbb Z}^n$-effective embedding in a 4-manifold being homeomorphic to $E(n)$.
This means that a cork gives a subset ${\mathbb Z}^n$ in the differential structures on $E(n)$.
Further, we describe handle decompositions of the twisted doubles (homotopy $S^4$) of Gompf's infinite order corks and show that they are Gluck twists and log transforms of $S^4$.
\end{abstract}

\section{Introduction}
\label{intro}
\subsection{Twist and cork.}
\label{Corktwist}
Let $X$ be a smooth manifold and $Y$ a codimension 0 submanifold with a smooth embedding $i:Y\hookrightarrow X$.
Removing $Y$ from $X$ and regluing by a self-diffeomorphism $f:\partial Y\to \partial Y$, we obtain a new smooth manifold and denote the manifold by $X(i,Y,f)$ or simply $X(Y,f)$.
We call the map $f$ a {\it gluing map} or a {\it twist map}.
This operation is called a {\it twist} and denote it by $(Y,f)$.
If the gluing map $f$ extends to $Y$ as a diffeomorphism, then we call $(Y,f)$ a {\it trivial twist}.

Let $Y$ be a contractible 4-manifold.
We call a nontrivial twist $(Y,f)$ a {\it cork}.
Then the gluing map $f$ is called a {\it cork map}.
Corks play significant roles in studying exotic 4-manifolds.
`Exotic' means that the manifolds are homeomorphic but non-diffeomorphic each other.
In fact, the following fact is well-known.
\begin{thm}[\cite{Mat}, \cite{CFHS}]
\label{corkoriginal}
Let $X,X'$ be two closed simply connected exotic 4-manifolds.
Then there exists a cork $(C,\tau)$ such that $X'=X(C,\tau)$ and $\tau^2=e$.
\end{thm}
\subsection{Gompf's infinite order corks}
\label{GCorktwist}
Gompf in \cite{G} gave an infinite exotic family using infinite order corks as below.
Let $X$ be a certain 4-manifold (including a square zero torus with two vanishing cycles).
\begin{thm}[\cite{G}]
\label{Gompf}
Suppose that $K_n$ is the $2n$-twist knot.
Then there exists an infinite order cork $(C,f)$ satisfying $X_{K_n}=X(C,f^n)$.
\end{thm}
$X_K$ is the (Fintushel-Stern) knot-surgery of $X$ by $K$.
The {\it cork order} of a cork is defined to be the minimal positive number whose power of the twist map is a trivial twist.
Higher order corks are known to exist by \cite{T} and \cite{AKMR}.

\begin{defn}[$G$-effective embedding (defined in \cite{AKMR})]
Let $\mathcal{C}$ be a 4-manifold with boundary.
Let $G$ be a group acting on $\partial\C$ effectively.
If there exists an embedding $i$ of $\C$ into a 4-manifold $X$ such that $X(i,\C,g)$ is not diffeomorphic to $X(i,\C,g')$ for any distinct $g,g'\in G$,
then we call the embedding $i$ a $G$-effective.
\end{defn}

\subsection{Galaxy}
Here we give terminologies related to cork and cork twist.
These terminologies make it easier to understand our results.
Let $X$ be a smooth 4-manifold.
We call the set of exotic structures on $X$ {\it the galaxy} of $X$ and denote it by
$$\g(X).$$
Our interest is to understand some kind of structures on the set $\g(X)$.
Any cork twist can be regarded as some relationship among subsets of $\g(X)$.

Let $C$ be a contractible 4-manifold.
Let $G$ be a nontrivial subgroup in $\text{Diff}(\partial C)$.
Let $(C,G)$ be a $G$-cork and $C\hookrightarrow X$ a $G$-effective embedding.
Then the collection $S=\{X(C,g)|g\in G\}$ is a subset of $\g(X)$ with one to one correspondence $g\mapsto X(C,g)$.
We call such a subset $S$ a {\it ($G$-)constellation} and the embedding $G\overset{\cong}{\to} S\subset \g(X)$ a {\it constellation embedding}.

Fact~\ref{corkoriginal} means that a pair of every two points in $\g(X)$ is a ${\mathbb Z}_2$-constellation.
If $G\hookrightarrow \g(X)$ is a constellation with respect to a $G$-cork $(C,G)$, then any subgroup $e\neq H<G$ gives an $H$-constellation $H\hookrightarrow \g(X)$ with respect to an $H$-cork $(C,H)$.
We call this constellation {\it a subconstellation}.
The main theorem in \cite{T5} says that any infinite family in $\g(X)$ is not always a constellation.
\subsection{Results.}
In this subsecion, we digest the results (in Section~\ref{Zncks} and \ref{S4}) obtained in this article.

The first result (Theorem~\ref{Z2}) gives a construction of $\Z^n$-cork.
Furthermore, we show that this cork gives a ${\mathbb Z}^n$-constellation in $\g(E(n))$
by knot-surgeries on a single fibered knot.
This construction is due to an $n$-fold boundary sum of Gompf's $C$.

The twisted double ${\mathbb S}_{r,s,m,k}$ of Gompf's infinite order corks $(C,f)$ is given by Gompf in \cite{G}.
We investigate the diffeomorphism type of ${\mathbb S}_{r,s,m,k}$.
We prove ${\mathbb S}_{r,s,m,k}\cong{\mathbb S}_{r,s,m+2,k}$.
We prove ${\mathbb S}_{r,s,0,k}$ is diffeomorphic to the standard $S^4$
and ${\mathbb S}_{r,s,1,k}$ is the Gluck twists of $S^4$ along a 2-knot in $S^4$.
These are the second result (Theorem~\ref{mod2prop}).
The third result (Theorem~\ref{double2}) is
a log transform construction of ${\mathbb S}_{r,s,m,k}$.
As a result, ${\mathbb S}_{r,s,m,k}$ is a $(1/s)$-log transform (or $(-1/r)$-log transform) along a torus in $S^4$.
As a result, ${\mathbb S}_{r,s,m,k}$ is a homotopy $S^4$ having three kinds of constructions:
a cork twist, a Gluck twist and a log transform.

\subsection{$\Z^n$-corks.}
\label{Zncks}
In \cite{G} Gompf defined infinite order corks $(C,f)$ and asked in \cite{G} whether there exists a ${\mathbb Z}^2$-cork by taking the full $T^2$ action of his corks.
In \cite{G2} he partially gave a negative answer for this question.
We construct a ${\mathbb Z}^n$-cork below, but it is not an answer of this question.
\begin{main}
\label{Z2}
For any natural number $n$ there exists a ${\mathbb Z}^n$-cork $C_n$.
Furthermore, there exist a 4-manifold $X_n$ (homeomorphic to $E(n)$) and ${\mathbb Z}^n$-effective embedding $C_n\hookrightarrow X_n$.
This ${\mathbb Z}^n$-effective embedding gives a ${\mathbb Z}^n$-constellation ${\mathbb Z}^n\hookrightarrow \g(E(n))$.
\end{main}
Here $C_n$ is the boundary sum of $n$ copies of $C(1,1;-1)$ as defined in \cite{G}.
This construction is due to performing cork twisting at distinct two clasps as mentioned by Gompf in \cite{G}.
Here we give the following interesting questions:
\begin{que}
Is there a 4-manifold $X$ such that $\g(X)$ 
includes a ${\mathbb Z}^n$-constellation for every natural number $n$?
\end{que}
\begin{que}
Does there exist a 4-manifold $X$ with a $G$-constellation in $\g(X)$ for an infinite non-abelian group?
\end{que}
Related topics to this question will be written in a sequel.

\subsection{The twisted double of Gompf's $(C,f)$.}
\label{S4}
Let $C$ denote $C(r,s;m)$.
The double $D(C)=C\cup_{\text{id}}(-C)$ is the boundary of $C\times I=B$.
Since, $C$ is a Mazur type (consisting of one 1-handle and one 2-handle) contractible 4-manifold, $B$ is diffeomorphic to the 5-ball.
Because, the attaching sphere of the 5-dimensional 2-handle is $S^1$ in 4-space,
it depends only on the homotopy type of $S^1$.
Thus, the two handles are a canceling pair.
In particular, the double $D(C)$ is diffeomorphic to $S^4$.

The twisted double $C\cup_{f^k}(-C)=S^4(C,f^k)=:{\mathbb S}_{r,s,m,k}$ 
are all homotopy 4-spheres.
As proven above, since the untwisted double ${\mathbb S}_{r,s,m,0}$ is $S^4$.
The problem that we concern is the diffeomorphism type of a general ${\mathbb S}_{r,s,m,k}$.
\begin{conj}[\cite{G}]
Let $r,s,m$ be any integers with  $r,s>0>m$.
Let $k$ be a nonzero integer.
Then ${\mathbb S}_{r,s,m,k}$ is standard $S^4$.
\end{conj}
We will prove the following theorem on ${\mathbb S}_{r,s,m,k}$ in Section \ref{section3}.
%We clarify the half of the diffeomorphism classes.
%The other half is the Gluck twist of $S^4$.
\begin{main}
\label{mod2prop}
For any integer $m$, ${\mathbb S}_{r,s,m,k}\cong {\mathbb S}_{r,s,m+2,k}$ holds.
If $m=0$, then ${\mathbb S}_{r,s,0,k}$ is diffeomorphic to $S^4$. If $m=1$, then ${\mathbb S}_{r,s,1,k}$ is the Gluck twist along a 2-knot $K_{r,s,k}$ in $S^4$
\end{main}
We give another aspect of ${\mathbb S}_{r,s,m,k}$.
\begin{main}
\label{double2}
${\mathbb S}_{r,s,m,k}$ is a $(1/s)$-log transform of $S^4$ along an embedded torus.
\end{main}
Note that exchanging the roles of $r$ and $s$, we also know that ${\mathbb S}_{r,s,m,k}$ is $(-1/r)$-log transform of $S^4$.
From the proof of Theorem \ref{double2} we get a handle decomposition of ${\mathbb S}_{r,s,m,k}$.
\begin{prop}
\label{handledecompositionofS4}
${\mathbb S}_{r,s,1,k}$ admits a handle decomposition with one $0$-handle, two 1-handles, four 2-handles, two 3-handles, and  one 4-handle.
\end{prop}

If the embedding of the torus in $S^4$ extends to a fishtail neighborhood, then the log transform does not change the diffeomorphism type as discussed in \cite{G2}.
Other similar situations appear in \cite{T2} and \cite{T3}.
In \cite{T3}, it is shown that the knot-surgery of $S^4$ along a torus is trivial by using a 
fishtail neighborhood embedding in $S^4$. 
It, however, seems difficult to find a certain fishtail neighborhood in ${\mathbb S}_{r,s,m,k}$.
Distinguishing the differential structures of ${\mathbb S}_{r,s,m,k}$ might be a hard work.

${\mathbb S}_{r,s,m,k}$ is also considered as a 4-manifold obtained by log transforms along two tori in $S^4$ as written in Section~\ref{twistedsec}.
The tori have a symmetry, then we have ${\mathbb S}_{r,s,m,k}\cong {\mathbb S}_{-s,-r,m,k}$
due to the handle diagram in {\sc Figure}~\ref{ZZk}.

Furthermore, the $0$-log transform and $(0,0)$-log transform of $S^4$ with respect to the torus and the tori are homotopy $S^3\times S^1\#S^2\times S^2$ and $\#^2S^3\times S^1\#^2S^2\times S^2$ respectively.
By the similar argument to the proof of Theorem~\ref{mod2prop}, if $m$ is even, then ${\mathbb S}_{\infty,s,m,k}\cong S^3\times S^1\#S^2\times S^2$ and ${\mathbb S}_{\infty,\infty,m,k}\cong \#^2S^3\times S^1\#^2S^2\times S^2$ hold.
In the case where $m$ is odd, ${\mathbb S}_{\infty,s,m,k}$ and ${\mathbb S}_{\infty,\infty,m,k}$ are
the Gluck twists along the standard 4-manifolds.
%This manifold is obtained by other log transforms.
%\begin{prop}
%${\mathbb S}_{r,s,m,k}$ is $s$ or $-r$ log transform of $S^3\times S^1\#S^2\times S^2$.
%or $(s,r)$-log transform of $\#^2S^3\times S^1\#^2S^2\times S^2$.
%\end{prop}
\begin{que}
Are $S^4[0]$ and $S^4[0,0]$ diffeomorphic to the
standard $(S^3\times S^1)\#(S^2\times S^2)$ or $\#^2(S^3\times S^1)\#^2(S^2\times S^2)$ respectively?
\end{que}
A similar construction is a {\it Scharlemann manifold} in \cite{S}, which is a surgery of $\Sigma\times S^1$ for a rational homology sphere $\Sigma$.
The surgeries are done along normally generating loops of $\pi_1(\Sigma)$ in $\Sigma\times S^1$.
In the case where $\Sigma$ is a Dehn surgery of a knot and the loop is the meridian of the knot,
the manifold is equivalent to a knot-surgery of the double of the fishtail neighborhood. For example see \cite{T3}.
The general Scharlemann manifold gives a homotopy
$$\#(S^3\times S^1)\#^l(S^2\times S^2)\#^m({\mathbb C}P^2\#\overline{{\mathbb C}P^2}).$$
The case where $\Sigma=\Sigma(2,3,5)=S^3_{-1}(\text{left handed }3_1)$ and the loop is the meridian of the trefoil is the original one in \cite{S}.
The author in \cite{T3} proved that some Scharlemann manifolds are standard.
Here we give another question.
\begin{que}
Are $S^4[0]$ and $S^4[0,0]$ diffeomorphic to a Scharlemann manifold and a connected-sum of those manifolds respectively?
\end{que}
\section*{Acknowledgements}
This study here was inspired by \cite{G} and partially done during my stay at KIAS on March in 2016.
I am deeply grateful for the hospitality at the institute, and Kyungbae Park and Min Hoon Kim.
Kouichi Yasui gave me much advice and suggestions in the Hiroshima University Topology-Geometry seminar.
Also, Robert Gompf gave me much useful comments and helps for this earlier manuscript of this paper.
I thank them so much.

\section{Gompf's infinite order cork $C$.}
\subsection{Knot-surgery}
Before drawing Gompf's infinite order corks, we review the definition of Fintushel-Stern knot-surgery.
Let $K$ be a knot in $S^3$.
Let $X$ be a 4-manifold with a square zero embedded torus $T$.
Then the operation 
$$X_K=[X-\nu(T)]\cup [(S^3-\nu(K))\times S^1]$$
is called a {\it (Fintushel-Stern) knot-surgery} along $K$.
The gluing map is defined in \cite{FS}.
The notation $\nu(\cdot)$ stands for a tubular neighborhood of a submanifold.
Then the Seiberg-Witten invariant formula of knot-surgery is the following:
$$SW_{X_K}=SW_X\cdot\Delta_K(t),$$
where $\Delta_K(t)$ is the Alexander polynomial of $K$.

We remark Sunukjian's work in \cite{Su}.
The paper says that if a 4-manifold $X$ has a non-trivial Seiberg-Witten invariant, then 
two knot-surgeries $X_{K_1}$ and $X_{K_2}$ along knots $K_1$ and $K_2$ with $\Delta_{K_1}\neq \Delta_{K_2}$ cannot be diffeomorphic.
In short, Alexander polynomial distinguishes smooth structures of
knot surgeries.
%To distinguish the smooth structures of knot-surgeries, we have only to consider the Alexander polynomial of the knots instead of Seiberg-Witten invariant.
\subsection{A handle diagram of $C$.}
In this section we describe the handle diagram of Gompf's infinite order cork $C$.
We call a handle diagram a {\it diagram} simply.
In \cite{G3} the diagram is described by the different method.
The manifold $C=C(r,s;m)$ can be built by attaching a single 2-handle to $I\times P$, where $P$ is the complement of $(r,s)$-double twist knot $\kappa(r,-s)$ as in {\sc Figure}~1 in \cite{G}.
\begin{lem}
The diagram of $C$ is {\sc Figure}~\ref{diagram} and the cork map $f$ is {\sc Figure}~\ref{diagram2}.
\end{lem}
\begin{figure}[htpb]
\begin{center}
\includegraphics{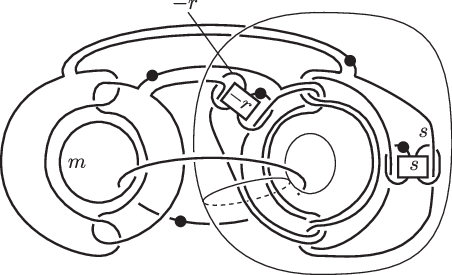}
\caption{The handle decomposition of $C(r,s;m)$.}
\label{diagram}
\end{center}
\end{figure}
\begin{figure}[htpb]
\begin{center}
\includegraphics[width=.9\textwidth]{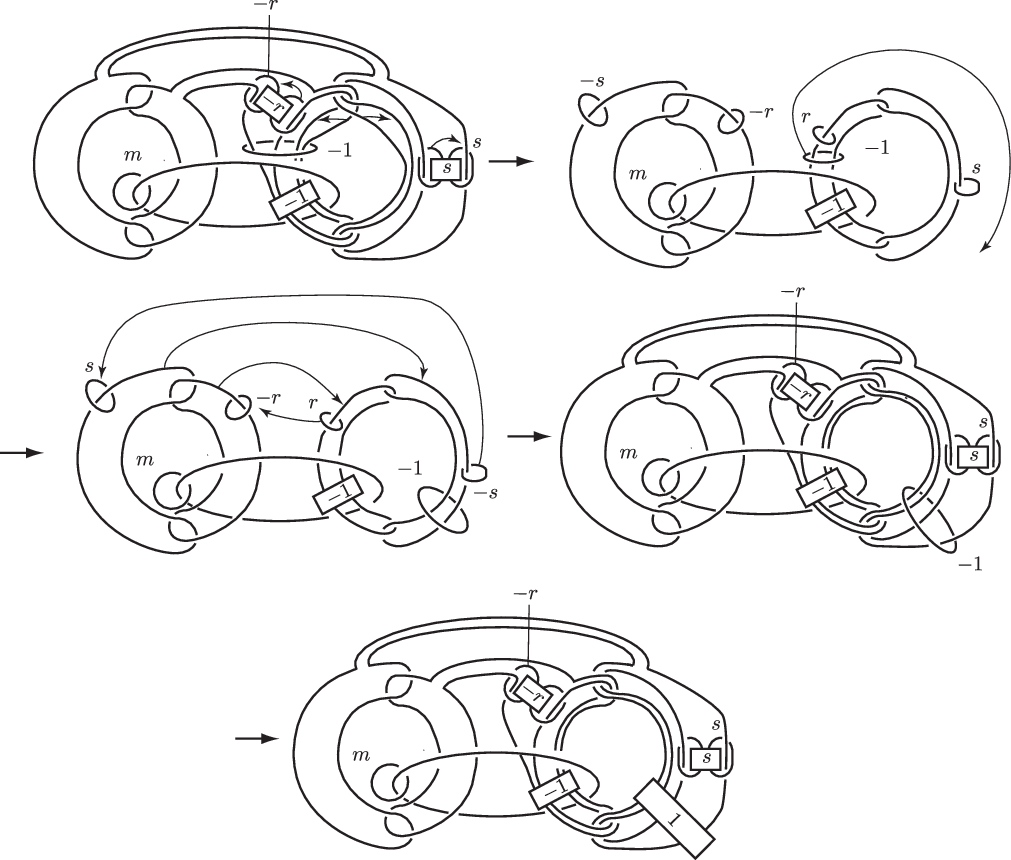}
\caption{The diffeomorphism $f$ on $\partial C(r,s;m)$.}
\label{diagram2}
\end{center}
\end{figure}

The boundary of Gompf's contractible 4-manifold $C$ has a torus decomposition along an incompressible torus.
The incompressible torus can be realized as a torus indicated in {\sc Figure}~\ref{diagram}.
Note that in the diagram in {\sc Figure}~\ref{diagram} the torus apparently cannot be realized as an embedded torus, because the torus meets the dotted 1-handles 4 times.
However, by inserting 2 pairs of canceling 2/3-handles, and putting the torus over the 2-handle, we can avoid the intersections.
As a result we can find our required embedded torus.

\begin{proof}
Let $\Sigma$ be a punctured torus.
$\Sigma\times S^1$ is diffeomorphic to the 0-surgered solid torus along the Bing double as in the left of {\sc Figure}~\ref{SigmaS}.
\begin{figure}[thpb]
\begin{center}
\includegraphics[width=.9\textwidth]{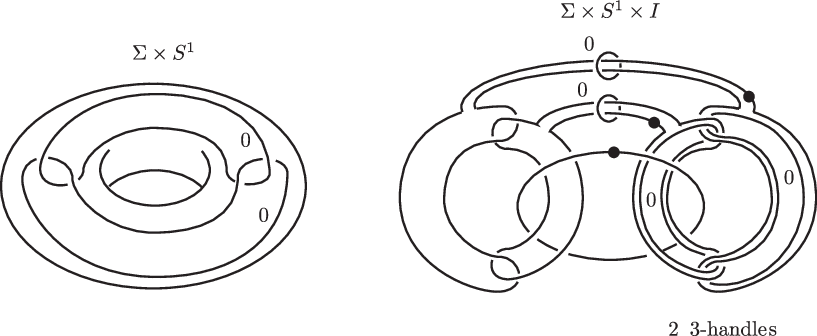}
\caption{$\Sigma\times S^1$ and $\Sigma\times S^1\times I$.}
\label{SigmaS}
\end{center}
\end{figure}

Thus by the fundamental method of handle calculus, the picture of the cylinder $I\times \Sigma\times S^1$ is the right diagram in {\sc Figure}~\ref{SigmaS}.
Attaching a $-r$-framed 2-handle (an $s$-framed 2-handle on another side respectively) and removing the union of the core and the attaching sphere cross interval $I$,
we obtain the other handle attachment and removal in {\sc Figure}~\ref{diagram}.
The two 0-framed 2-handles in the right of {\sc Figure}~\ref{SigmaS} are canceled out with the 3-handles when removing the cores.
The similar situation appears in \cite{T2}.

Hence, the handle diagram of $C$ is as in the picture in {\sc Figure}~\ref{diagram}.
Reducing the diagram, we get a ribbon 1-handle and $m$-framed 2-handle along the meridian of the ribbon knot.\qed
\end{proof}
The map $f$ is defined to be the right handed Dehn twist cross the identity on
$$(I\times \partial \Sigma)\times S^1.$$
This cork $C$ produces a knot-surgery $X_K$ for a 4-manifold $X$ for a square zero torus $T$.
The torus satisfies the following condition.

We assume $r=s=-m=1$.
Let $V$ be the neighborhood of the Kodaira's singular fibration of type III, which $V$ is also used in \cite{T4} in the same situation.
Suppose that $V$ is embedded in $X$.
$T$ is embedded as a general fiber of $V\subset X$.
Then there exists an embedding $C\subset V\subset X$ such that 
$$X_{K}=X(C,f^k),$$
where $K=\kappa(k,-1)$ is the $2k$-twist knot.

If $C$ does not satisfy $r=s=-m=1$, we can construct the knot surgery by the twist knot on an embedded torus under certain condition as mentioned in \cite{G}.

Due to \cite{G}, for any integer $k$ the $k$-fold composition $f^k$ cannot extend to the inside $C$ as any diffeomorphism.

We describe in {\sc Figure}~\ref{imagef} the local deformation of $X(C,f^2)$ according to the diffeomorphism $f$ in {\sc Figure}~\ref{diagram2}.
\begin{figure}[htpb]
\begin{center}
\includegraphics[width=.9\textwidth]{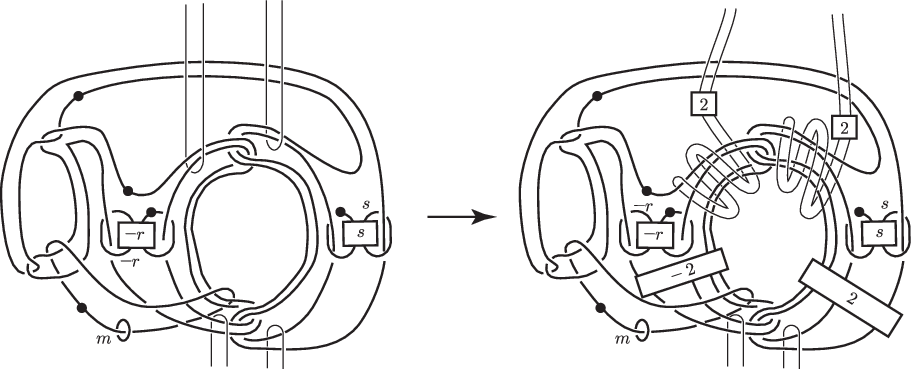}
\caption{The cut and paste of $C(r,s;m)$ by $f^2$.}
\label{imagef}
\end{center}
\end{figure}

\subsection{2-bridge knots $K_{m,n}$}
In the next section we prove Theorem~\ref{Z2}.
First we prepare a 2-bridge knot $K_{m,n}$ as in {\sc Figure}~\ref{2b} for integers $m,n$.
\begin{figure}[htbp]
\begin{center}
\includegraphics{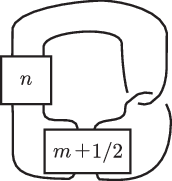}
\caption{$K_{m,n}$.}
\label{2b}
\end{center}
\end{figure}
The knot $K_{m,n}$ is classified as follows:
\begin{equation}
K_{m,n}=\begin{cases}T_{2,2m-1}&n=0\\\text{unknot} &(m,n)=(-1,-1)\\T_{2,-3}&(m,n)=(0,1)\\\text{non-torus 2-bridge knot}&\text{otherwise.}\end{cases}
\end{equation}
The following equality holds
$$K_{-1,n}\approx K_{0,n+1}.$$
Here $T_{p,q}$ is the right handed $(|p|,|q|)$-torus knot if $pq>0$ and is the left handed $(|p|,|q|)$-torus knot if $pq<0$.

Let $\Delta_{m,n}$ denote the Alexander polynomial $\Delta_{K_{m,n}}$.
The Alexander polynomial is computed as follows:
If $m\ge 1$, then
$$\Delta_{m,n}(t)\doteq n(t^{m+1}+t^{-m-1})-3n(t^m+t^{-m})+(4n+1)\sum_{i=-m+1}^{m-1}(-1)^{i-m+1}t^i$$
and if $m=0$, then 
$$\Delta_{0,n}(t)=n(t+t^{-1})-(2n-1).$$
This formula can be easily proven by using the skein relation of the Alexander polynomial.

We prove the following lemma:
\begin{lem}
If $m\ge1$, then the polynomials $\Delta_{m,n}(t)$ are distinct each other in ${\mathbb Z}[t,t^{-1}]/\pm t^{\pm1}$.
In particular, for $(m,n)\neq (m',n')$ with $m,m'\ge 1$, $K_{m,n}$ and $K_{m',n'}$ are non-isotopic as unoriented knots.
\end{lem}
\begin{proof}
Suppose that $\Delta_{m,n}=\Delta_{m',n'}$.
If $n\neq 0$ and $n'\neq0$ hold, then, comparing the coefficients of top degree terms, we have $n=n'$ and $m=m'$.
If either $n$ or $n'$ is 0 and the other is not 0, then the two polynomials do not agree.
Suppose that $n'\neq 0$ and $n=0$ hold.
Then we have $\Delta_{m,n}(t)=\Delta_{T_{2,2m-1}}(t)$.
Comparing the top degrees of $\Delta_{m,n}(t)$ and $\Delta_{m',n'}(t)$, we have $n'=\pm1$.
However, comparing the values of the second top degree of $\Delta_{m',n'}(t)$ 
we must have $n'=\pm3$.
This is contradiction.
If $n=n'=0$, then, comparing the top degrees of the two polynomials, we have $m=m'$.
\hfill$\Box$
\end{proof}

\subsection{Proof of Theorem \ref{Z2}.}
We remark that the natural number $k$ in this proof corresponds to $n$ in the statement.

Let $K(n_1,\cdots,n_k)$ be $K_{1,n_1}\#\cdots\#K_{k,n_k}$.
Hence $K(0,\cdots,0)=\#_{i=1}^kT_{2,2i-1}$ holds.
We embed $k$ copies of $\Sigma$ in the exterior $E$ of $K(0,\cdots,0)$ disjointly.
In the case of 2 copies see {\sc Figure}~\ref{disjointZZ}.
\begin{figure}[thpb]
\begin{center}
\includegraphics{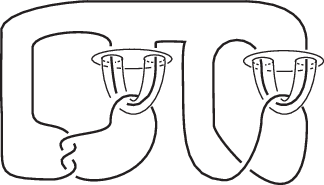}
\caption{A disjoint embedding of two copies of $\Sigma$ in $E_2$.}
\label{disjointZZ}
\end{center}
\end{figure}
Thus, disjoint $k$ copies of $I\times \Sigma\times S^1$ in $E\times S^1$ are also embedded.
By attaching $2k$ ($-1$-framed) 2-handles on the meridians of $E\times S^1$ and $k$ ($-1$-framed) 2-handles on the $S^1$-direction, we can embed $k$ copies of $C=C(1,1;-1)$ in $X_k:=E(k)_{\#_{i=1}^kT_{2,2i-1}}$, because it has $12k$ vanishing cycles.
%Removing the $2k$ union of the core disks in the 2-handles and attaching spheres times $I$, we have $k$ disjoint embedded $C$'s in $X_k$.

We take each point in the complements of the incompressible tori in $\partial C\cup \partial C$
and connect the two components by a 1-handle attached on the neighborhoods of the two points in $X_k$, which the 1-handle is disjoint from $k$ $C$'s.
Embedding such $k-1$ 1-handles, we construct $\natural^kC\hookrightarrow X_k$.

Let $f_i$ be a diffeomorphism on $\partial(\natural^kC)$ which acts as Gompf's $f$ on the $i$-th component of $\#^k\partial C$ and acts as the identity on the other component of $\#^k\partial C$.
Since those points are taken in the complement of the incompressible tori,
the two maps $f_i$ and $f_j$ are commutative.
The twist $(\natural^kC, f_1^{n_1}\cdots f_k^{n_k})$ of $X_k$ produces $E(k)_{K(n_1,\cdots,n_k)}$.

Then, the computation of the Seiberg-Witten invariants is as follows:
$$SW_{E(k)_{K(n_1,\cdots,n_k)}}=SW_{E(k)}\prod_{i=1}^k\Delta_{i,n_i}$$
$$SW_{X_k}=SW_{E(k)}\prod_{i=1}^k\Delta_{i,0}.$$
Comparing the degrees of the two results, the two Seiberg-Witten invariants do not agree, unless $n_i=0$ for all $i$.
Since $X_k$ and $E(k)_{K(n_1,\cdots,n_k)}$ are exotic when $(n_1,\cdots,n_k)\neq (0,\cdots,0)$.
$(\natural^kC, f_1^{n_1}\cdots f_k^{n_k})$ gives an exotic $E(k)$.
Thus, $(\natural^kC, f_1^{n_1}\cdots f_k^{n_k})$ is a cork.
This means that $(\natural^kC, \{f_1^{n_1}\cdots f_k^{n_k}|n_j\in {\mathbb Z}\})$ is a ${\mathbb Z}^k$-cork.

To prove that this embedding is a ${\mathbb Z}^k$-effective embedding, we have only to show that if $\prod_{i=1}^k\Delta_{i,n_i}(t)=\prod_{i=1}^k\Delta_{i,n_i'}(t)$, then $(n_1,\cdots,n_k)=(n_1',\cdots,n_k')$.
\begin{claim}
\label{2.3}
Let $k, p$ be integers with $k>0$ and $p\ge 0$ and $n_i,n_i'(i=1,\cdots, k)$ integers.
If we have
\begin{equation}
\label{eq1}
\prod_{i=1}^{k-p}\Delta_{p+i,n_{p+i}}=\prod_{i=1}^{k-p}\Delta_{p+i,n_{p+i}'},
\end{equation}
then we have $n_{p+1}=n_{p+1}'$.
\end{claim}
If $\Delta_{p,q}$ were irreducible, then this claim would be easy.
However, since some 2-bridge knots are ribbon, 
such Alexander polynomials are not always irreducible.

\begin{proof}
By the induction of the number $k$ in (\ref{eq1}) we prove this claim.
Let $\sigma_i$ and $\sigma_i'$ be the $i$-th elementary symmetric polynomials in $n_{p+1},\cdots, n_{k}$ and $n_{p+1}',\cdots, n_{k}'$ respectively.
For example, we see 
$$\sigma_i=\sum_{\{\ell_1,\cdots, \ell_i\}\subset \{p+1,\cdots, k\},\#\{\ell_1,\cdots, \ell_i\}=i}n_{\ell_1}\cdots n_{\ell_i}.$$

Let $d$ be the degree of (\ref{eq1}).
Comparing the degree $d$ of (\ref{eq1}), we obtain 
$$\sigma_{k-p}=\sigma_{k-p}'.$$
Let $j_0$ be an integer with $1\le j_0\le p+1$.
Comparing the coefficients of the degree $d-2j$ of (\ref{eq1}) with $0\le j\le j_0$ we have $\sigma_{k-p-j_0}=\sigma_{k-p-j_0}'$.

Further, the coefficients with degree $d-2p-3$ of the left hand side of (\ref{eq1}) is 
\begin{eqnarray*}
&&S+(-3n_{p+1})\prod_{j=p+2}^{k-p}n_j+\sum_{j=p+2}^{k-p}(-4n_j-1)\prod_{\overset{\ell=p+1}{\ell\neq j}}^kn_\ell\\
&=&S_0-\sum_{j=p+2}^{k}\prod_{\overset{\ell=p+1}{\ell\neq j}}^kn_{\ell}=S_0-\sigma_{k-p-1}+n_{p+2}\cdots n_{k},
\end{eqnarray*}
where $S,S_0$ are polynomials generated by $\sigma_{k-p},\sigma_{k-p-1}\cdots,\sigma_{k-p-j}$.

Thus $n_{p+2}\cdots n_{k}=n_{p+2}'\cdots n_{k}'$ holds.
By using $\sigma_{k-p}=\sigma_{k-p}'$, we have $n_{p+1}=n_{p+1}'$.
\qed
\end{proof}
We go back to the proof of Theorem \ref{Z2}.
Suppose that $\prod_{i=1}^k\Delta_{i,n_i}=\prod_{i=1}^k\Delta_{i,n_i'}$.
Then, by using Claim \ref{2.3} in the case of $p=0$ we have $n_1=n_1'$.
By dividing $\Delta_{1,n_1}$ from the both side of this equality, we have $\prod_{i=2}^k\Delta_{i,n_i}=\prod_{i=2}^k\Delta_{i,n_i'}$.
Iterating this process by using Claim \ref{2.3}, we have $n_2=n_2'$, $\cdots$, $n_k=n_k'$ holds.

Thus $\prod_{i=1}^k\Delta_{i,n_i}=\prod_{i=1}^k\Delta_{i,n_i'}$ implies $(n_1,\cdots,n_k)=(n_1',\cdots, n_k')$.
Therefore we show that this embedding 
$$\natural^kC:=C_k\hookrightarrow X_k=E(k)_{\#^k_{i=1}T_{2,2i-1}}$$
is ${\mathbb Z}^k$-effective.
\qed

This proof focuses on the case of $C=C(1,1;-1)$.
However, Gompf's method in Remarks (a) in \cite{G} would change our examples to $C=C(r,s;m)$ with $r,s>0>m$.
\section{The twisted double ${\mathbb S}_{r,s,m,k}$ of $C$.}
\label{section3}
\subsection{The diagrams of twisted doubles.}
\label{twistdouble}
Let ${\mathbb S}_{r,s,m,k}$ denote the homotopy $S^4$ defined in Section~\ref{S4}.
\label{twistedsec}
We prove the following proposition first of all.
\begin{prop}
\label{double}
In the case of $k=1$, the diagram of ${\mathbb S}_{r,s,m,1}$ is {\sc Figure}~\ref{ZZ}.
The handle diagram of ${\mathbb S}_{r,s,m,k}$ is {\sc Figure}~\ref{ZZk}.
\end{prop}
\begin{proof}
The move of the meridians of 2-handles by $f$ is described in {\sc Figure}~\ref{movef}.
\begin{figure}[htpb]
\begin{center}
\includegraphics[width=.9\textwidth]{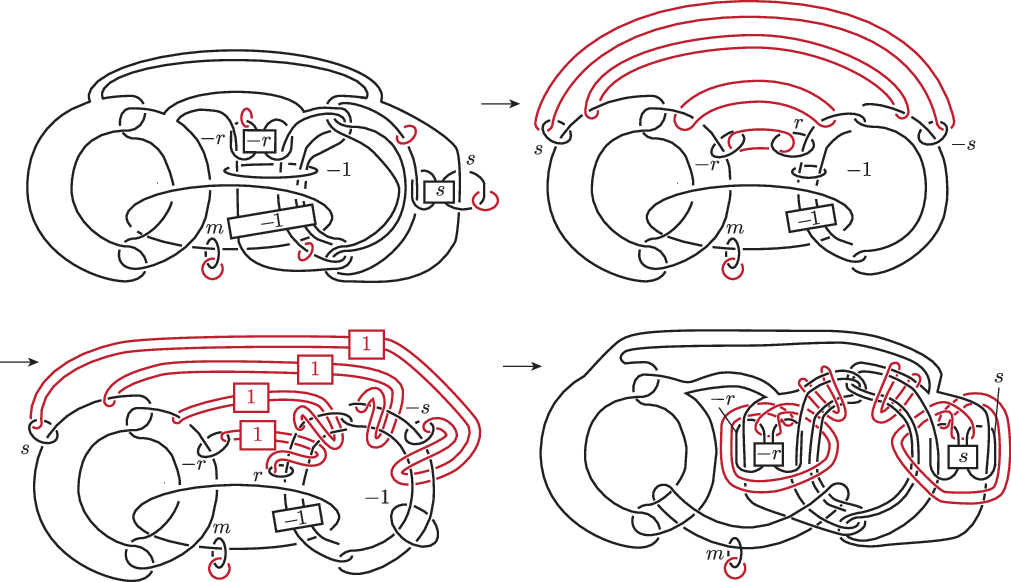}
\caption{The move of meridians of 2-handles by the $f$.}
\label{movef}
\end{center}
\end{figure}
Hence by $f$, the meridians of four 2-handles in {\sc Figure}~\ref{diagram} are moved to the link $\alpha,\beta, \gamma$ and $\delta$ in {\sc Figure}~\ref{ZZ}.

Hence, this diagram in {\sc Figure}~\ref{ZZ} is the $k=1$ case.
In the general $k$ case, the images of the meridians of the 2-handles by $f^k$ are $\alpha,\beta,\gamma$ and $\delta$ in {\sc Figure}~\ref{ZZk}.
%Other 0-framed 2-handles $\alpha',\beta',\gamma',\delta'$ in {\sc Figure}~\ref{ZZk} are additional 2,3-canceling pairs to introduce $f^k$.
Thus the first diagram in {\sc Figure} \ref{ZZk} describes ${\mathbb S}_{r,s,m,k}$.
%The components $\alpha',\beta',\gamma',\delta'$ in the boundary of 2-handlebody of ${\mathbb S}_{r,s,m,k}$ are isotopic to $0$-framed unlink by the easy 3-dimensional surgery description calculus.
%The four components cancel out together with 4 3-handles.
The $\beta,\gamma$ and the meridian of the $m$-framed knot can be canceled with two 3-handles, because the calculation in {\sc Figure}~\ref{ul} proves that
the three components are the unlink.
\begin{figure}[htbp]
\begin{center}
\includegraphics[width=.9\textwidth]{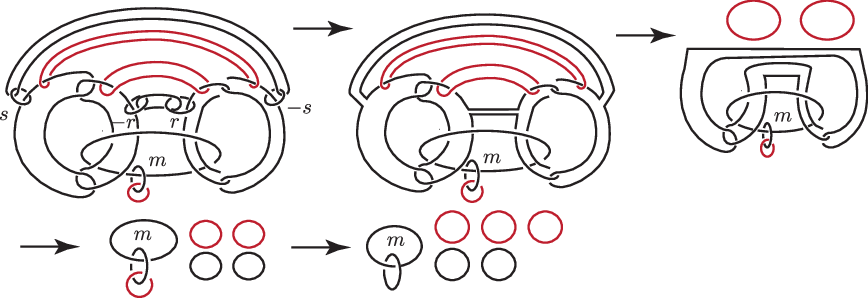}
\caption{This calculation proves that the red three components is the unlink.}
\label{ul}
\end{center}
\end{figure}
After the canceling, $\beta,\gamma$ and the meridian of the $m$-framed 2-handle are the $0$-framed unlink on the handle decomposition of the 2-handlebody.
Thus we cancel those together with three 3-handles.
Then we obtain the second picture as in {\sc Figure}~\ref{ZZk}.
\end{proof}

\begin{figure}[thpb]
\begin{center}
\includegraphics[width=.65\textwidth]{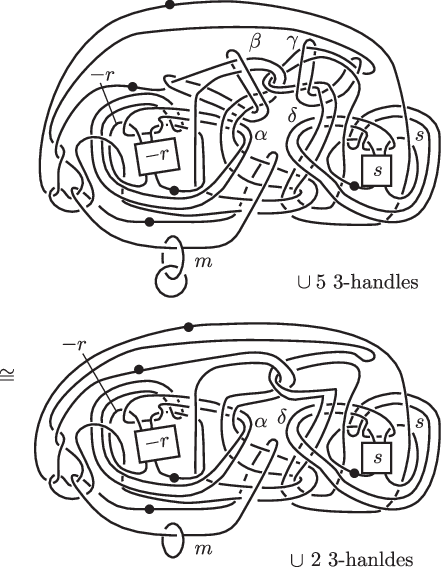}
\caption{Twisted double ${\mathbb S}_{r,s,m,1}$.}
\label{ZZ}
\end{center}
\end{figure}

\begin{figure}[thpb]
\begin{center}
\includegraphics[width=.8\textwidth]{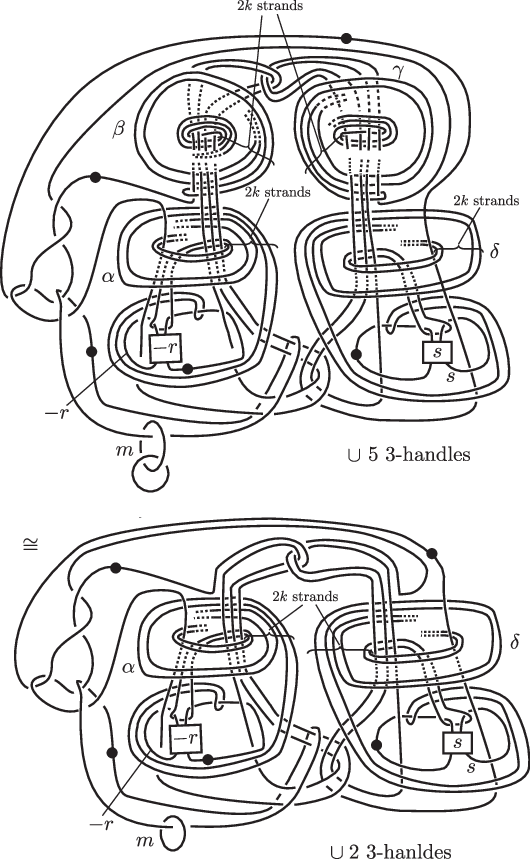}
\caption{Twisted double ${\mathbb S}_{r,s,m,k}$ and a diagram after canceling.}
\label{ZZk}
\end{center}
\end{figure}
\qed

{\bf Proof of Proposition \ref{handledecompositionofS4}.}
The union of 2-handles $\alpha$, $\delta$, two 3-handles and a 4-handle is diffeomorphic to $C$.
Hence ${\mathbb S}_{r,s,m,k}=C\cup_{f^k}(-C)$ admits two 1-handles and four 2-handles.
\hfill$\Box$

In particular, in {\sc Figure} \ref{ZZk}, the four 2-handles are $\alpha,\delta$, $s$-framed 2-handle and ($-r$)-framed 2-handle.

If one proves that ${\mathbb S}_{r,s,m,k}$ is standard, the first diagram in {\sc Figure} \ref{ZZk} would be useful as an auxiliary information of the handle decomposition.

Let $X$ be a 4-manifold.
We consider an embedded $S^2\hookrightarrow X$ with trivial normal bundle.
The homotopy types of self-diffeomorphisms on $S^2\times S^1$ is isomorphic to ${\mathbb Z}/2{\mathbb Z}$.
Let $\tau$ be a self-diffeomorphism over $S^2\times S^1$ with non-trivial homotopy class.
The surgery $X\leadsto X(S^2\times D^2,\tau)$ is called a {\it Gluck twist}.
\subsection{Proofs of Theorem \ref{mod2prop}.}
First we prove that ${\mathbb S}_{r,s,0,k}$ is diffeomorphic to $S^4$.
The deformation of the pictures in {\sc Figure}~\ref{standard} presents the ${\mathbb S}_{r,s,0,k}\cong S^4$.

In the diagram of ${\mathbb S}_{r,s,0,k}$ the union of the (lowest) $0$-framed 2-handle and a 4-handle consists of $S^2\times D^2$.
By removing the $S^2\times D^2$ and attaching an $m$-framed 2-handle and a 4-handle, we obtain the surgery $S^4={\mathbb S}_{r,s,0,k}\leadsto {\mathbb S}_{r,s,0,k}(S^2\times D^2,\tau^m)={\mathbb S}_{r,s,m,k}$.
Thus ${\mathbb S}_{r,s,m,k}$ is the Gluck surgery of $S^4$.
Hence, the diffeomorphism ${\mathbb S}_{r,s,m,k}\cong {\mathbb S}_{r,s,m-2,k}$ holds.
\hfill$\qed$\\
The last diffeomorphism can be also verified by the calculus in {\sc Figure} \ref{modmod}.
\begin{figure}[thpb]
\begin{center}
\includegraphics[width=.7\textwidth]{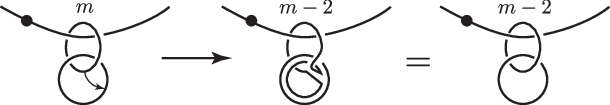}
\caption{The diffeomorphism ${\mathbb S}_{r,s,m,k}\cong {\mathbb S}_{r,s,m-2,k}$.}
\label{modmod}
\end{center}
\end{figure}

\subsection{The log transform along several tori.}
Here we define the log transform along $n$ tori with the square zero.
Let $e_i:T_i^2\hookrightarrow X$ be disjoint embedded tori each other that the squares are all zero for $i=1,\cdots n$.
Let $c_i$ be a curve presenting a primitive element in $H_1(T^2_i)$.
Let $p_i,q_i$ be several pairs of coprime integers.
Let $\tilde{e}_i$ be an embedding of the tubular neighborhood of $T^2$ with respect to $e_i$.
Suppose that a gluing diffeomorphism $g_{c_i,p_i,q_i}:T^2\times \partial D^2\to T_i^2\times \partial D^2$ a diffeomorphism satisfying 
$$\partial D^2\mapsto p_i\cdot \partial D^2+q_i\cdot c_i.$$
In fact, the image of $\partial D^2$ by the gluing map is the attaching sphere of the unique 2-handle in $T^2\times D^2$ and the remaining handles of $T^2\times D^2$ are two 3-handles and one 4-handle.
The information of the image of $\partial D^2$ determines a diffeomorphism type.
Hence, the diffeomorphism type of the log transform along the tori depends only on $e_i, c_i, p_i$ and $q_i$.

We denote the log transform by 
$$X[(e_i),(c_i),(p_i/q_i)].$$
and call {\it $(p_i/q_i)$-log transform along $(e_i)$ with direction $(c_i)$}.
When embeddings $(e_i)$ and curves $(c_i)$ are clear in the context, we omit these items.

\subsection{Proof of Theorem~\ref{double2}.}
First, we find $T^2\times D^2$ in the diagrams in {\sc Figure}~\ref{ZZ} (${\mathbb S}_{r,s,m,1}$), in general, {\sc Figure} \ref{ZZk} (${\mathbb S}_{r,s,m,k}$).
Removing handles, we can show that the submanifold as in {\sc Figure}~\ref{ZZ5} is embedded in ${\mathbb S}_{r,s,m,k}$.
It contains two disjoint $T^2$s in the last picture.
\begin{figure}[thpb]
\begin{center}
\includegraphics[width=.9\textwidth]{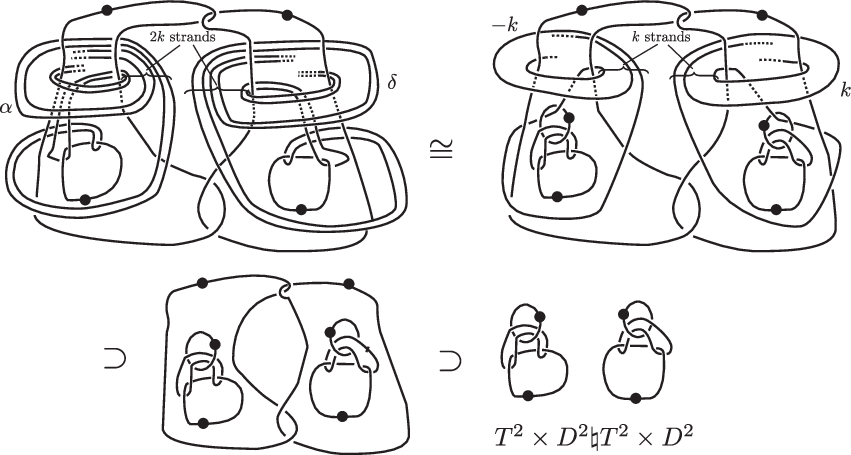}
\caption{Disjoint embedded two $T^2$'s in $S^4$.}
\label{ZZ5}
\end{center}
\end{figure}
%
%\begin{figure}[thpb]
%\begin{center}
%\includegraphics[width=.7\textwidth]{twisted3.eps}
%\caption{$T^2\times D^2\cup(\text{1-handle $\cup$ 2-handle})$ in ${\mathbb S}_{r,s,m,1}$.}
%\label{ZZ3}
%\end{center}
%\end{figure}

One time $(1/1)$-log transform with direction $\ell$ corresponds to the change of the diagram which is given in {\sc Figure}~\ref{ZZ4}.
In general, the $(1/s)$-log transform is the $s$-times iteration of this process.
\begin{figure}[thpb]
\begin{center}
\includegraphics[width=.9\textwidth]{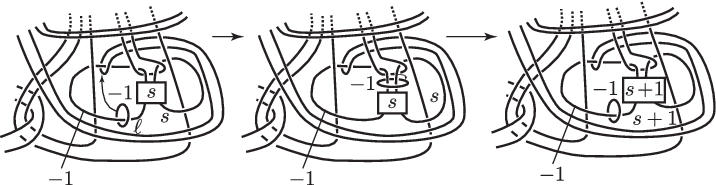}
\caption{The $(1/1)$-log transform with direction $\ell$.}
\label{ZZ4}
\end{center}
\end{figure}
Hence, ${\mathbb S}_{r,s,m,k}$ is the $(1/s)$-log transform of ${\mathbb S}_{r,0,m,k}$.
Since the knot $\kappa(r,0)$ is isotopic to the unknot, $C(r,0;m)$ is the standard 4-ball.
Thus ${\mathbb S}_{r,0,m,k}$ is diffeomorphic to $S^4$.

Therefore, ${\mathbb S}_{r,s,m,k}$ is a $(1/s)$-log transform along a torus embedded in $S^4$.
By exchanging the roles of $r$ and $s$, we can show that ${\mathbb S}_{r,s,m,k}$ is a $(-1/r)$-log transform along another torus in $S^4$.
\qed

%In the diagram in {\sc Figure} \ref{ZZ5} the embedding of the two tori used in the proof of Theorem \ref{double2} is described.

Theorem \ref{double2} says that the twisted double ${\mathbb S}_{r,s,m,k}$ is obtained by two log transforms along two embedded disjoint tori in $S^4$ as in {\sc Figure}~\ref{ZZ5}.
Namely, we have
$${\mathbb S}_{r,s,m,k}=S^4[(e_{m,k,1},e_{m,k,2}),(c_1,c_2),(-1/r,1/s)]=S^4[(-1/r,1/s)].$$
The two torus embeddings $e_{m,k,i}:T^2_i\hookrightarrow S^4$ ($i=1,2$) are embedded in such a way that each torus is embedded in each component $T^2\times D^2$ in $\natural^2T^2\times D^2$.
The $T^2\times D^2\natural T^2\times D^2$ exterior in $S^4$ is described in {\sc Figure} \ref{ZZk2}.
\begin{figure}[thpb]
\begin{center}
\includegraphics[width=.65\textwidth]{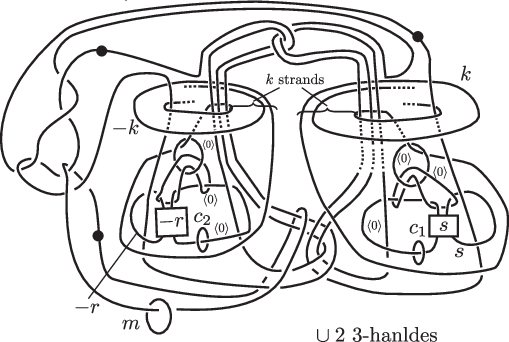}
\caption{The $\natural^2T^2\times D^2$ exterior in $S^4$ and curves $c_1,c_2$ on the boundary.}
\label{ZZk2}
\end{center}
\end{figure}

\subsection{A remark for the curves $c_1$ and $c_2$.}
As mentioned in Section \ref{S4}, if either of curves $c_1$ or $c_2$ in the boundary has an embedded slice disk in the exterior with framing $-1$.
However, it is difficult to find such a disk in terms of the following observation.

We consider a cobordism $\mathcal{C}$ from $\#^2T^2\times S^1$ to $\#^3S^2\times S^1$ by removing the three 3-handles and one 4-handle from the exterior which is described in {\sc Figure} \ref{ZZk2}.
We take an annulus in $\mathcal{C}$ beginning from either of $c_1$ or $c_2$.
Suppose that the annulus has no critical points in $\mathcal{C}$, i.e., the annulus is the trace by the gradient flow of the Morse function for the handle decomposition.
Let $\tilde{c}$ be the obtained knot in $\#^3S^2\times S^1$.
Turning the union of three 3-handles and a 4-handle by the upside down calculus, we obtain a knot description in {\sc Figure} \ref{isotopyc}.
Clearly, this knot $\tilde{c}$ has no $-1$-framed disk in $\natural^3D^3\times S^1$.
Because if there is such a disk, then by attaching a 2-handle on $\tilde{c}$ with $0$-framing, we must find a $(-1)$-sphere in the attached manifold whose intersection form is $\langle 0\rangle$.
This is a contradiction.

This means that the easy way cannot be found a $-1$-framed embedded disk in the exterior for $e_{m,k,i}$.
\begin{figure}[thpb]
\begin{center}
\includegraphics[width=.3\textwidth]{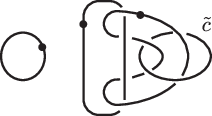}
\caption{The isotopy class of $c_1\subset \partial (\natural^3D^3\times S^1)$.}
\label{isotopyc}
\end{center}
\end{figure}
\subsection{The cases of $r$ or $s=\infty$.}

The cases where either $r$ or $s$ is $\infty$ can be regarded as $0$-log transforms along the tori from the equality ${\mathbb S}_{r,s,m,k}=S^4[-1/r,1/s]$.
Namely, ${\mathbb S}_{\infty ,s,m,k}=S^4[0,1/s]$ and ${\mathbb S}_{\infty,\infty,m,k}=S^4[0,0]$.

${\mathbb S}_{\infty,\infty,m,k}$ is obtained by exchanging two dots and two 0's in the sub-handle for $\natural^2T^2\times D^2$.
The diagram of ${\mathbb S}_{\infty,\infty,m,k}$ is described in {\sc Figure}~\ref{std}.
By computing the fundamental groups and homology groups, the manifolds ${\mathbb S}_{\infty,0,m,k}$ and ${\mathbb S}_{\infty,\infty,m,k}$ are homotopic to 
$$(S^3\times S^1)\#(S^2\times S^2)\text{ and }\#^2(S^3\times S^1)\#^2(S^2\times S^2)$$
respectively.
In the case of $m\equiv 0\bmod 2$, by using the same move as {\sc Figure}~\ref{standard},
we have ${\mathbb S}_{\infty,s,m,k}=(S^3\times S^1)\#(S^2\times S^2)$
and ${\mathbb S}_{\infty,\infty,m,k}=\#^2(S^3\times S^1)\#^2(S^2\times S^2)$.
Thus essentially, we have to investigate the manifold ${\mathbb S}_{\infty,s,1,k}$ and ${\mathbb S}_{\infty,\infty,1,k}$.
These are the Gluck twists of standard $(S^3\times S^1)\#(S^2\times S^2)$ or $\#^2(S^3\times S^1)\#^2(S^2\times S^2)$.

In \cite{T3} some homotopy $(S^3\times S^1)\#(S^2\times S^2)$'s are constructed according to Scharlemann's method.
Some of those are diffeomorphic to the standard manifold.
What one knows the relationship between these manifolds would be an interesting problem.
\begin{figure}[thpb]
\begin{center}
\includegraphics[width=.9\textwidth]{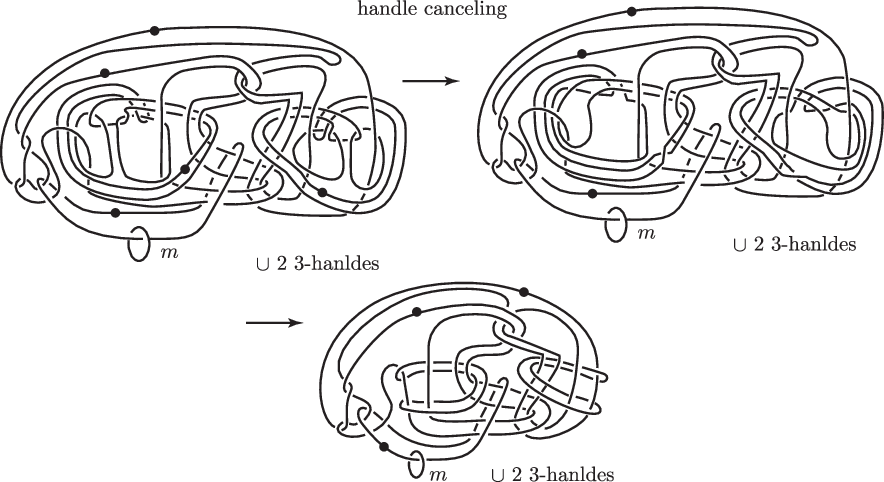}
\caption{The handle diagram of $S^4[0,0]$.}
\label{std}
\end{center}
\end{figure}
%\subsection{}
%These tori obtain other potential exotic manifolds.
%Let $S^4_0$, $S^4_{0,0}$ denote the $(0,\infty)$-log transform and $(0,0)$-log transform along the two tori.
%$S^4_0$, and $S^4_{0,0}$ give homotopy
%$$(S^3\times S^1)\#(S^2\times S^2)$$
%and 
%$$\#^2(S^3\times S^1)\#^2(S^2\times S^2)$$
%respectively.

  %%%% 参考文献
  \begin{figure}[thpb]
\begin{center}
\includegraphics[width=.65\textwidth]{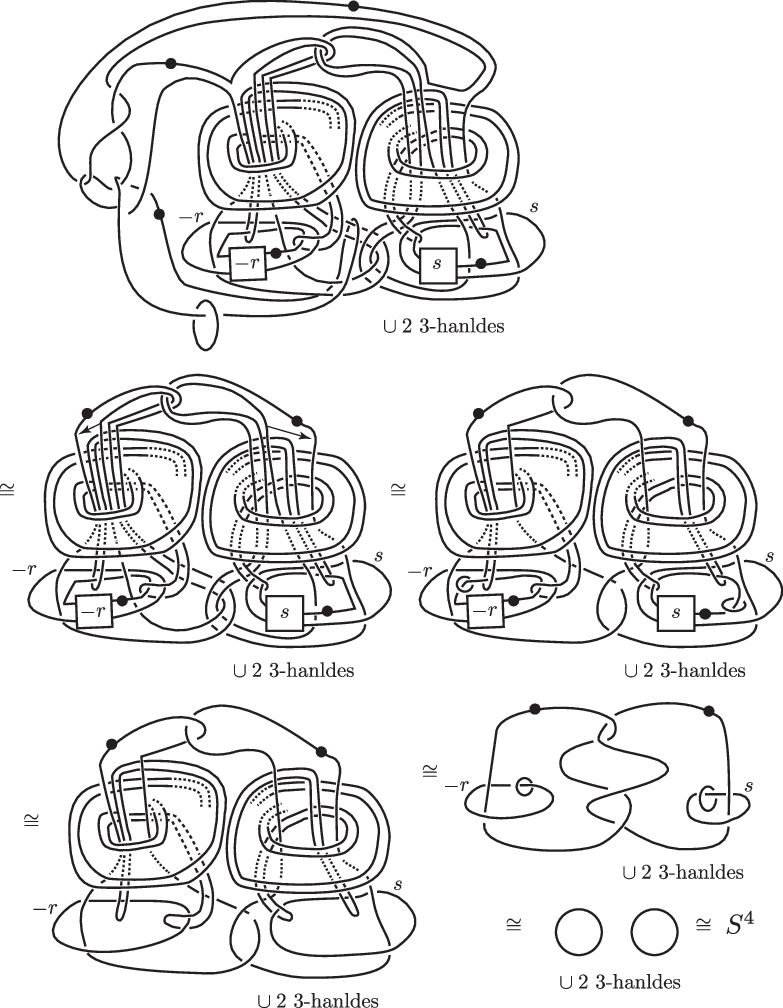}
\caption{${\mathbb S}_{r,s,0,k}$ is the standard $S^4$.}
\label{standard}
\end{center}
\end{figure}

\end{document}